\documentclass[11pt]{amsart}
\usepackage[latin9]{inputenc}
\usepackage{color}
\usepackage{float}
\usepackage{units}
\usepackage{amstext}
\usepackage{amsthm}
\usepackage{amssymb}
\usepackage{graphicx}

\makeatletter

\providecommand{\tabularnewline}{\\}
\floatstyle{ruled}
\newfloat{algorithm}{tbp}{loa}
\providecommand{\algorithmname}{Algorithm}
\floatname{algorithm}{\protect\algorithmname}

\numberwithin{equation}{section}
\numberwithin{figure}{section}
\theoremstyle{plain}
\newtheorem{thm}{\protect\theoremname}

\usepackage{algorithmic} 
\usepackage{algorithm} 
\usepackage[numbers,sort,compress]{natbib}
\usepackage{graphicx}

\DeclareMathOperator*{\argmin}{arg\,min}

\theoremstyle{definition}

\numberwithin{equation}{section}

\title[SUPERIORIZING PRECONDITIONED CONJUGATE GRADIENT ALGORITHMS]{SUPERIORIZATION OF PRECONDITIONED CONJUGATE GRADIENT ALGORITHMS FOR TOMOGRAPHIC IMAGE RECONSTRUCTION}

\author[Helou, E. S.]{Elias S. Helou}

\address[Helou, E. S.]{Institute of Mathematical Science and Computation, University of S\~ao Paulo, SP, Brazil}
\email{{\tt elias@icmc.usp.br}}

\author[Herman, G. T.]{Gabor T. Herman}
\address[Herman, G. T.]{Department of Computer Science, The Graduate Center, City University of New York, NY, USA}

\author[Lin, C.]{Chuan Lin}
\address[Lin, C.]{School of Electrical Engineering, Southwest Jiaotong University, Chengdu, China}

\author[Zibetti, M. V. W.]{Marcelo V. W. Zibetti}
\address[Zibetti, M. V. W.]{Program in Electrical and Computer Engineering, Federal University of Technology, PR, Brazil}



\@ifundefined{showcaptionsetup}{}{%
 \PassOptionsToPackage{caption=false}{subfig}}
\usepackage{subfig}
\makeatother

\providecommand{\theoremname}{Theorem}

\begin{document}
\begin{abstract}
Properties of Superiorized Preconditioned Conjugate Gradient (SupPCG)
algorithms in image reconstruction from projections are examined.
Least squares (LS) is usually chosen for measuring data-inconsistency
in these inverse problems. Preconditioned Conjugate Gradient algorithms
are fast methods for finding an LS solution. However, for ill-posed
problems, such as image reconstruction, an LS solution may not provide
good image quality. This can be taken care of by superiorization.
A superiorized algorithm leads to images with the value of a secondary
criterion (a merit function such as the total variation) improved
as compared to images with similar data-inconsistency obtained by
the algorithm without superiorization. Numerical experimentation shows
that SupPCG can lead to high-quality reconstructions within a remarkably
short time. A theoretical analysis is also provided.
\end{abstract}

\maketitle

\section{Introduction}

Superiorization~\cite{Herman2012,hzm16,chj17,gah14,gh17} is an algorithmic
framework that consists in perturbing iterative methods. It can be
applied to feasibility problems (problems of finding a point belonging
to a given set) or to optimization problems (problems of finding the
best point among a set of candidates, according to some criterion).
Its aim is to improve the iterates, according to a secondary criterion
(which is some sort of merit function), produced by an iterative method
for solving the original problem without the secondary criterion.

Many problems in science or engineering can be mathematically modeled
as feasibility~\cite{com96} or optimization$\,$\cite{HERMAN:2009a}
problems. In this paper we use tomographic imaging as our example
application of the superiorized algorithms that we introduce, but
the basic idea is useful in numerous other contexts.\footnote{The reader can find a continuously updated bibliography about superiorization
at http://math.haifa.ac.il/YAIR/bib-superiorization-censor.html.}

We investigate a superiorized Preconditioned Conjugate Gradient~(PCG)
method, applying it to tomographic imaging and comparing it to other
approaches to this application. In Section~\ref{sec:opttomo} we
connect tomographic image reconstruction to optimization, as in~\cite{HERMAN:2009a}.
Section~\ref{sec:spcg} reviews superiorization, the Conjugate Gradient~(CG)
and the PCG algorithms. Superiorization of these algorithms results
in the algorithms Superiorized Conjugate Gradient~(SupCG) and Superiorized
Preconditioned Conjugate Gradient~(SupPCG), and superiorization of
one of the Algebraic Reconstruction Techniques~(ART)~\cite[ Chapter 11]{HERMAN:2009a}
leads to SupART. Section~\ref{sec:numerical} reports on numerical
experiments comparing the performance in image reconstruction of these
algorithms, as well as a version of ART that uses blob basis functions~\cite[Section 6.5]{HERMAN:2009a}.
Section~\ref{sec:Theorethical-Analysis} presents a theoretical analysis
of the SupPCG algorithm. Finally, Section~\ref{sec:conclusions-1}
provides the concluding remarks.

\section{Optimization in tomographic reconstruction}

\label{sec:opttomo}

Assume that the image $x\in\mathbb{R}^{n}$ generates the data $b\in\mathbb{R}^{m}$
through
\[
Rx=b_{exact}+e,
\]
where $R\in\mathbb{R}^{m\times n}$ is the projection (also called
Radon) matrix, \textbf{$b=b_{exact}+e$,} and $e\in\mathbb{R}^{m}$
is the unknown error vector introduced into the measurements. Therefore,
we need to solve $Rx\approx b\text{,}$ where the meaning of the approximation
must be well defined mathematically. There are several ways of doing
this, in the present paper we follow the least squares approach:
\begin{equation}
x\in\argmin_{y\in\mathbb{R}^{n}}\|Ry-b\|^{2}\text{.}\label{eq:least_squares}
\end{equation}
Because the least squares approach is well accepted and extremely
general, it is not surprising that several techniques have been developed
for its solution. The minimizers in (\ref{eq:least_squares}) are
given by the solutions of the normal equations
\begin{equation}
R^{T}Rx=R^{T}b\text{,}\label{eq:normal_equations}
\end{equation}
which is a linear system of equations. Due to the huge size and sparsity
of the projection matrix (as can be the case in image reconstruction),
iterative methods that successively approximate a solution may be
the most practical.

Among the best-known iterative techniques for the general problem
of least squares, one can point at CG and at PCG, to be applied to
the system~\eqref{eq:normal_equations}; see \cite{Vogel-2002}.
For tomographic image reconstruction, algebraic reconstruction techniques~(ART)~\cite[ Chapter 11]{HERMAN:2009a}
are widely applied.

To put this into a general context, let $f$ be a function defined
over images such that $f(x)$ is some measure of the inconsistency
of $x$ with the given data. The primary aim of a proposed reconstruction
method should be to produce an $x$ for which the value of $f$ is
relatively small, therefore $f$ is referred to as the primary criterion.
Here we concentrate on the squared error, $f(x)=\|Rx-b\|^{2}$, as
the primary criterion. We also discuss a primary criterion based on
a Bayesian approach to image reconstruction \cite{hhl79}, namely
\begin{equation}
r^{2}\left\Vert Rx-b\right\Vert ^{2}+\left\Vert x-\mu_{X}\right\Vert ^{2},\label{eq:Bayesian}
\end{equation}
where the number $r$ is the so-called signal-to-noise ratio and $\mu_{X}$
is a uniformly gray image with the gray value being the average value
of the image as can be accurately estimated based on all the measured
data \cite[Section 6.4]{HERMAN:2009a}. Both these need to be specified
for a particular reconstruction task.

\section{Superiorization of PCG and ART}

\label{sec:spcg}

\subsection{Superiorization}

As stated in \cite{chj17}:
\begin{quote}
The superiorization methodology is used for improving the efficacy
of iterative algorithms whose convergence is resilient to certain
kinds of perturbations. Such perturbations are designed to ``force''
the perturbed algorithm to produce more useful results for the intended
application than the ones that are produced by the original iterative
algorithm. The perturbed algorithm is called the ``superiorized version''
of the original unperturbed algorithm. If the original algorithm is
computationally efficient and useful in terms of the application at
hand and if the perturbations are simple and not expensive to calculate,
then the advantage of this method is that, for essentially the computational
cost of the original algorithm, we are able to get something more
desirable by steering its iterates according to the designed perturbations.
\end{quote}
Several iterative algorithms use an updating of the current approximation
to the problem at hand that is of the form
\[
\left(x_{k+1},u_{k+1}^{1},\dots,u_{k+1}^{\nu}\right)=\mathit{\mathrm{U}}\left(x_{k},u_{k}^{1},\dots,u_{k}^{\nu}\right),
\]
where the $u_{k}^{i}$ for $1\leq i\leq\nu$ are some auxiliary vectors
and $x_{k}$ is the image at iteration $k$. A superiorized version
of such a method can have the form
\begin{equation}
\left(x_{k+1},u_{k+1}^{1},\dots,u_{k+1}^{\nu}\right)=\mathit{\mathrm{U}}\left(x_{k+\nicefrac{1}{2}},u_{k}^{1},\dots,u{}_{k}^{\nu}\right)\text{,}\label{eq: sup_step}
\end{equation}
where $x_{k+\nicefrac{1}{2}}=x_{k}+s_{k}$, with $\{s_{k}\}$ referred
to as a superiorization sequence. For our experiments in Section \ref{sec:numerical}
we used the method from~\cite{Herman2012}, as specified in Algorithm~\ref{algo:superiorization},
for generating the superiorization sequence. The actual values produced
by that specification depend on the chosen secondary criterion and
the procedure to compute ``nonascending vectors'' for this criterion;
here we report on results for total variation ($TV$) as the secondary
criterion (see \cite{Herman2012,gah14,gh17}).\textcolor{blue}{{} }

\textcolor{black}{Total variation is defined for a vector $y\in\mathbb{R}^{n}$
that represents a discrete image $[y]$ with $l$ rows and $c$ columns
of pixels with $lc=n$. The value $[y]_{i,j}$ of the pixel in the
$i$th row and $j$th column of this image is defined to be the $(c(i-1)+j)$th
component of vector $y$. The definition of total variation is
\[
TV(y):=\sum_{i=1}^{l-1}\sum_{j=1}^{c-1}\sqrt{\left([y]_{(i,j)}-[y]_{(i+1,j)}\right)^{2}+\left([y]_{(i,j)}-[y]_{(i,j+1)}\right)^{2}}.
\]
}

\textcolor{black}{Algorithm~\ref{algo:superiorization} uses nonascending
vectors. Following~\cite{Herman2012}, given a point $y\in\mathbb{R}^{n}$,
we say that $t\in\mathbb{R}^{n}$ is a nonascending vector for $TV$
at $y$ if $\|t\|\leq1$ and there is $\Delta>0$ such that
\[
\delta\in[0,\Delta]\Rightarrow TV(y+\delta t)\leq TV(y).
\]
We denote the set of nonascending vectors for $TV$ at $y$ as $\tilde{\partial}TV(y)$.
Note that $\tilde{\partial}TV(y)$ is never empty because the zero
vector is in $\tilde{\partial}TV(y)$. In order to get the $t\in\tilde{\partial}TV(y)$
that we actually use, first define $\overline{t}\in\mathbb{R}^{n}$
componentwise. For any component of $\overline{t}$, its value is
the negative of the partial derivative of $TV$ at $y$ with respect
to that component, if this partial derivative is well defined, and
is $0$ otherwise. We finally define the nonascending vector $t\in\tilde{\partial}TV(y)$
that we use as
\[
t=\begin{cases}
\frac{\overline{t}}{\|\overline{t}\|}, & \text{if }\left\Vert \overline{t}\right\Vert \neq0,\\
\mathrm{the\:zero\:vector}, & \text{otherwise.}
\end{cases}
\]
That the $t$ specified this way is indeed a nonascending vector for
$TV$ at $y$ is an immediate consequence of~\cite[Theorem 2]{Herman2012}.}

The specification\textcolor{blue}{{} }\textcolor{black}{of Algorithm~}\textcolor{blue}{\ref{algo:superiorization}
}involves further parameters: $K\in\mathbb{N}$ (the number of nonascending
steps), $a\in(0,1)$ (the step-length diminishing factor for each
nonascending trial) and $\gamma\in(0,\infty)$ (the starting step-length);
for complete specification, these parameters need to be selected by
the user. Next we provide details of the purpose and operation of
Algorithm 1.

\begin{algorithm}[t]
\caption{$\mathrm{S}_{TV}(x,\ell,a,\gamma,K)$}

\label{algo:superiorization}

\normalsize{

\begin{algorithmic}[1]

\STATE{$y\leftarrow x$}

\STATE{\textbf{for $i=1,2,\dots,K$}}

\STATE{$\quad$$t\in\tilde{\partial}TV(y)$}

\STATE{$\quad$\textbf{repeat}}

\STATE{$\quad$$\quad$$\tilde{\gamma}\leftarrow\gamma a^{\ell}$}

\STATE{$\quad$$\quad$$\tilde{y}\leftarrow y+\tilde{\gamma}t$}

\STATE{$\quad$$\quad$$\ell\leftarrow\ell+1$}

\STATE{$\quad$\textbf{until} $TV(\tilde{y})\leq TV(y)$}

\STATE{$\quad$$y\leftarrow\tilde{y}$}

\STATE{$\textbf{return}$ $(y-x,\ell)$}

\end{algorithmic}}
\end{algorithm}

Algorithm~\ref{algo:superiorization} specifies the $s_{k}$ to be
used to define $x_{k+\nicefrac{1}{2}}=x_{k}+s_{k}$, which is then
further used to provide us with the next image $x_{k+1}$ of the superiorized
iterative algorithm, see (\ref{eq: sup_step}). As the iterations
proceed, changes are made not only to the image vector but also to
an integer variable $\ell$ (see Step 7 of Algorithm 1), these changes
are needed to make the superiorized algorithm behave as desired (see
\cite{Herman2012,zlh17} and Section \ref{sec:Theorethical-Analysis}
below). In the superiorized algorithm, $\ell$ is initialized to be
$0$ (that is why $\gamma$ is referred to as the starting step-length).
In each step~\eqref{eq: sup_step} of the superiorized iterative
algorithm, Algorithm 1 is called with the current values of the image
vector $x_{k}$, the integer $\ell$ and the user-specified parameters
$a$, $\gamma$ and $K$ (these do not change during an execution
of the superiorized algorithm). Algorithm~\ref{algo:superiorization}
returns $s_{k}\left(=y-x_{k}\right)$ and the new value of $\ell$,
to be used in the next iterative step of the algorithm. During the
execution of the superiorized algorithm, the step lengths $\tilde{\gamma}$
form a summable sequence as needed in the theoretical results on the
behavior of superiorized algorithms (see, for example, \cite[Theorem 1]{Herman2012}).

Superiorization has, by design, an influence over algorithmic convergence,
but ideally it maintains the required properties of the limit point.
The influence of the superiorization should be that the sequence of
iterates $\{x_{k}\}$ (or $\{x_{k+\nicefrac{1}{2}}\}$) for the superiorized
method has reduced secondary criterion (in our case, $TV$) when compared
to the equivalent sequence produced by the unsuperiorized method.
For our method this is true because the superiorization step, although
only guaranteed not to increase the value of the undifferentiable
$TV$ criterion, is frequently capable of actually reducing the value
of the $TV$. That is, one usually finds that $TV(x_{k+\nicefrac{1}{2}})<TV(x_{k})$.

We provide a theoretical discussion of the convergence of the SupPCG
algorithm in Section~\ref{sec:Theorethical-Analysis}. A convergence
proof of the SupCG algorithm appears in \cite{zlh17} and that result
is extended naturally to the superiorized version of the PCG algorithm,
which is a method that we now proceed to discuss.

\subsection{Preconditioned Conjugate Gradient\label{sub:Preconditioned-Conjugate-Gradien}}

The CG algorithm is designed to solve a linear system of equations
$Ax=y$, where $A$ is symmetric positive semi-definite. If $A=R^{T}R$
and $y=R^{T}b$, then CG can solve~\eqref{eq:least_squares} via~\eqref{eq:normal_equations}.
The CG method may converge slowly depending on the distribution of
the eigenvalues of the system matrix $A$. The idea of PCG is to replace
the system $Ax=y$ by a preconditioned system $MAx=My$, where $M$
is a symmetric positive-definite matrix, so that the two systems have
the same solutions but $MA$ has better spectral properties for the
application of the CG algorithm. We return to this topic later in
Section~\ref{sec:Theorethical-Analysis}.

\begin{algorithm}[t]
\caption{$\mathrm{U}_{\text{PCG}}$$(x,p,h)$}

\label{algo:PCG_step}

\begin{algorithmic}[1]

\STATE{\textbf{$g\leftarrow Ax-y$}}

\STATE{\textbf{$z\leftarrow Mg$}}

\STATE{\textbf{$\beta\leftarrow z^{T}h/p^{T}h$}}

\STATE{\textbf{$p\leftarrow-z+\beta p$}}

\STATE{$h\leftarrow Ap$}

\STATE{\textbf{$\alpha\leftarrow-g^{T}p/p^{T}h$}}

\STATE{\textbf{$x\leftarrow x+\alpha p$}}

\STATE{\textbf{return }$(x,p,h)$}

\end{algorithmic}
\end{algorithm}

Given $\mathrm{U}_{\text{PCG}}$ defined by Algorithm~\ref{algo:PCG_step},
the update for PCG algorithm is
\[
\left(x_{k+1},p_{k},h_{k}\right)=\mathrm{U_{\text{PCG}}}\left(x_{k},p_{k-1},h_{k-1}\right)\text{,}
\]
where $x_{0}\in\mathbb{R}^{n}$ is an arbitrary vector, but $x_{1}$,
$p_{0}$ and $h_{0}$ must follow precise rules, as described for
SupPCG in the next subsection. The preconditioning matrix $M$ that
we use here is of the form $F^{-1}DF$, where $F$ is the Discrete
Fourier Transform (DFT) and $D$ is a diagonal matrix that represents
a generalized Hamming window \cite{HERMAN:2009a} (using parameters
$\mu$ and $\rho$) combined with a ramp filter for frequency $\omega$
with $-\pi\leq\omega\leq\pi$:
\begin{equation}
h(\omega)=(|\omega|+\mu)\cdot(\rho+(1-\rho)\cos\omega).\label{eq:preconditioning}
\end{equation}

\subsection{Superiorized Preconditioned Conjugate Gradient\label{sub:Superiorized-Preconditioned-Conj}}

In this subsection we describe, as Algorithm~\ref{algo:SPCG}, the
SupPCG method, which is the main subject of the present paper. In
its specification, we use the notations $f(x)=\|Rx-b\|^{2}$, $A=R^{T}R$
and $y=R^{T}b$; see (\ref{eq:least_squares}) and (\ref{eq:normal_equations}).
The inputs of Algorithm~\ref{algo:SPCG} are the initial image $x_{0}$,
the $a$, $\gamma$ and $K$ that have the same roles as in Algorithm~\ref{algo:superiorization}
and a user-specified $\varepsilon>0$ that determines the termination
of Algorithm~\ref{algo:SPCG} (see Step 10). The algorithm returns
the number of iterations $k$ at the time of its termination and an
output image $x_{k-\nicefrac{1}{2}}$.

\begin{algorithm}[t]
\caption{SupPCG$\left(x_{0},a,\gamma,K,\varepsilon\right)$}

\label{algo:SPCG}

\normalsize{

\begin{algorithmic}[1]

\STATE{\textbf{$x_{\nicefrac{1}{2}}\leftarrow x_{0}$}}

\STATE{\textbf{$g_{0}\leftarrow Ax_{0}-y$}}

\STATE{\textbf{$z_{0}\leftarrow Mg_{0}$}}

\STATE{\textbf{$p_{0}\leftarrow-z_{0}$}}

\STATE{\textbf{$h_{0}\leftarrow Ap_{0}$}}

\STATE{\textbf{$\alpha\leftarrow-g_{0}^{T}p_{0}/p_{0}^{T}h_{0}$}}

\STATE{\textbf{$x_{1}\leftarrow x_{0}+\alpha p_{0}$}}

\STATE{\textbf{$k\leftarrow1$}}

\STATE{\textbf{$\ell_{1}\leftarrow0$}}

\STATE{\textbf{$\textbf{while}$ $f\left(x_{k-\nicefrac{1}{2}}\right)>\varepsilon$}}

\STATE{$\quad$\textbf{$(s_{k},\ell_{k+1})\leftarrow\mathrm{S}_{TV}\left(x_{k},\ell_{k},a,\gamma,K\right)$}}

\STATE{\textbf{$\quad$$x_{k+\nicefrac{1}{2}}\leftarrow x_{k}+s_{k}$}}

\STATE{$\quad$$(x_{k+1},p_{k},h_{k})\leftarrow\mathrm{U}_{\text{PCG}}\left(x_{k+\nicefrac{1}{2}},p_{k-1},h_{k-1}\right)$}

\STATE{$\quad$$k\leftarrow k+1$}

\STATE{$\textbf{return }\left(k,x_{k-\nicefrac{1}{2}}\right)$}

\end{algorithmic}}
\end{algorithm}

For the termination of Algorithm~\ref{algo:SPCG} to be guaranteed
by the results we present in Section~\ref{sec:Theorethical-Analysis},
we will need appropriate bounds on the elements of the superiorization
sequence $\left\{ s_{k}\right\} $; we now provide such bounds.
\begin{thm}
\textcolor{black}{Upon the execution of Step 11 in Algorithm~\ref{algo:SPCG},
we have that $\|s_{k}\|\leq K\gamma a^{(k-1)K}$. }\end{thm}
\begin{proof}
\textcolor{black}{To verify this assertion, we first note that the
vectors $t\in\tilde{\partial}TV(y)$ in Algorithm 1 satisfy $\|t\|\leq1$,
by the definition of nonascending vectors as given above. Next we
prove inductively that $\ell_{k}\geq(k-1)K$. This inequality holds
for $k=1$ as an equality. Now assume, for the induction, that it
holds for iteration $k$. The value of $\ell_{k+1}$ is determined
by Step 11 of Algorithm~\ref{algo:SPCG}, which consists in a call
$\mathrm{S}_{TV}\left(x_{k},\ell_{k},a,\gamma,K\right)$ to Algorithm~\ref{algo:superiorization}.
In each of the $K$ loops of Algorithm~\ref{algo:superiorization}
(see its Step 2), $\ell$ is increased at least once (in Step 7).
This and the induction hypothesis implies that for the $\ell_{k+1}$
returned by $\mathrm{S}_{TV}\left(x_{k},\ell_{k},a,\gamma,K\right)$,
we have $\ell_{k+1}\geq\ell_{k}+K\geq\left(k-1\right)K+K=kK=\left((k+1\right)-1)K;$
this completes the inductive proof.}

\textcolor{black}{We introduce some extra notations. Consider the
call $\mathrm{S}_{TV}(x_{k},\ell_{k},a,\gamma,K)$ to Algorithm~\ref{algo:superiorization}
in Algorithm~\ref{algo:SPCG}. In the execution of that call, there
is an inner loop (indexed by $i$) consisting of Steps 3-9 of Algorithm~\ref{algo:superiorization}.
We now use $t_{k,i}$ to denote the value of $t$ obtained by Step
3 in that inner loop and $\ell_{k,i}$ as $\ell-1$ at the time when
the condition $TV(\tilde{y})\leq TV(y)$ in Step 8 of the inner loop
is satisfied. Looking at $\mathrm{S}_{TV}\left(x_{k},\ell_{k},a,\gamma,K\right)$,
we see that for the $s_{k}$ returned by Step 11 of Algorithm~\ref{algo:SPCG}
we have that $s_{k}=\gamma\sum_{i=1}^{K}a^{\ell_{k,i}}t_{k,i}$. Since
$\left\Vert t_{k,i}\right\Vert \leq1$, $\ell_{k,i}\geq\ell_{k}\geq(k-1)K$
and $a\in(0,1)$, we conclude that $\left\Vert s_{k}\right\Vert \leq\gamma\sum_{i=1}^{K}\left\Vert a^{\ell_{k,i}}t_{k,i}\right\Vert \leq\gamma\sum_{i=1}^{K}a^{\ell_{k,i}}\left\Vert t_{k,i}\right\Vert \leq\gamma\sum_{i=1}^{K}a^{\ell_{k}}\leq K\gamma a^{(k-1)K}$.}
\end{proof}

\subsection{Superiorized Algebraic Reconstruction Technique}

The Algebraic Reconstruction Techniques consist of sequentially projecting
the current image toward the hyperplanes defined by the equations
of a linear system and we consider a full cycle through the equations
to be a single iteration. Several variants are possible, we use the
version originally proposed in \cite{hhl79} with pixel basis functions
(we refer to this as ART) and also with blob basis functions \cite{mhc98},
see alternatively \cite[Section 6.5]{HERMAN:2009a} (we refer to this
as \break{}BlobART). We also report on superiorization using the
$TV$ secondary criterion for the pixel basis, we refer to the so
superiorized ART as SupART.

To be more specific, the ART method proposed in \cite{hhl79} (see
alternatively \cite[Section 11.3]{HERMAN:2009a}) aims at finding
the minimizer of the primary criterion given in (\ref{eq:Bayesian}).
Therefore the signal-to-noise ratio $r$ is a parameter that needs
to be specified for both ART and for BlobART. There is an additional
parameter $\lambda$ for these algorithms, it is the so-called relaxation
parameter that determines the size of the step in projecting the current
image toward the next hyperplane. Mathematical convergence theory
allows the relaxations to vary with iterations \cite{HERMAN:2009a,hhl79},
and this may be useful in practice, but here we choose a single $\lambda$
for a reconstruction run.

\section{Numerical experimentation}

\label{sec:numerical}

\subsection{Tomographic data acquisition simulation}

We used data obtained from simulation of a phantom that includes features
such as seen in medical images \cite[Fig. 4.6(a)]{HERMAN:2009a}.
The phantom has low-contrast features of interest inside a high attenuation
skull-like structure. If shown in its full grayscale range, the inner
features are nearly invisible, therefore the phantom and reconstructions
are displayed using a gray scale in which linear attenuation values
smaller than $0.204\;\text{cm}^{-1}$ are displayed as black and values
higher than $0.21675\:\text{cm}^{-1}$ appear as white. Phantom and
tomographic data were obtained using SNARK14.\footnote{SNARK14 may be downloaded free of charge from http://turing.iimas.unam.mx/SNARK14M/.}
(SNARK14 is the most recent version of a series of software packages
for the reconstruction of 2D images from 1D projections. From the
point of view of the current paper, its main advantage over the previous
version SNARK09 \cite{KLUK13} is the new ability to superiorize automatically
any iterative algorithm.) The phantom can be seen at the bottom-right
of Fig.~\ref{fig:phantom_and_best_images}. It consists of an array
of $243\times243$ pixels, each of which corresponds to an area of
$0.0752\times0.0752\;\text{cm}^{2}$. Data were collected for $360$
equally spaced angles in $[0,\pi)$ and each of these angles was sampled
for $345$ parallel rays spaced at $0.0752\;\text{cm}$ between them.
Random photon emission at the x-ray source and scattering of the x-ray
photons was simulated.

\begin{figure}
\begin{centering}
\begin{minipage}[t]{0.5\columnwidth}%
\subfloat[PCG, iteration 2]{\centering{}%
\begin{minipage}[t]{1\columnwidth}%
\begin{center}
\includegraphics[width=0.7\textwidth]{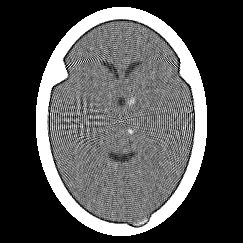}
\par\end{center}%
\end{minipage}}%
\end{minipage}%
\begin{minipage}[t]{0.5\columnwidth}%
\subfloat[SupPCG, iteration 3]{\centering{}%
\begin{minipage}[t]{1\columnwidth}%
\begin{center}
\includegraphics[width=0.7\textwidth]{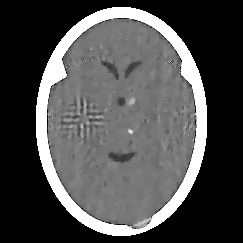}
\par\end{center}%
\end{minipage}}%
\end{minipage}
\par\end{centering}

\begin{centering}
\begin{minipage}[t]{0.5\columnwidth}%
\subfloat[CG, iteration 12]{\centering{}%
\begin{minipage}[t]{1\columnwidth}%
\begin{center}
\includegraphics[width=0.7\columnwidth]{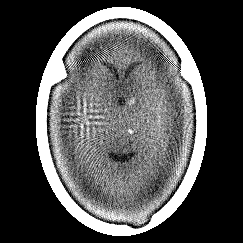}
\par\end{center}%
\end{minipage}}%
\end{minipage}%
\begin{minipage}[t]{0.5\columnwidth}%
\subfloat[SupCG, iteration 35]{\centering{}%
\begin{minipage}[t]{1\columnwidth}%
\begin{center}
\includegraphics[width=0.7\columnwidth]{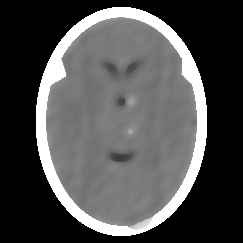}
\par\end{center}%
\end{minipage}}%
\end{minipage}
\par\end{centering}

\begin{centering}
\begin{minipage}[t]{0.5\columnwidth}%
\subfloat[ART, iteration 15]{\centering{}%
\begin{minipage}[t]{1\columnwidth}%
\begin{center}
\includegraphics[width=0.7\textwidth]{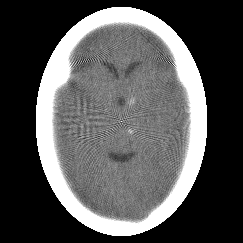}
\par\end{center}%
\end{minipage}}%
\end{minipage}%
\begin{minipage}[t]{0.5\columnwidth}%
\subfloat[SupART, iteration 12]{\centering{}%
\begin{minipage}[t]{1\columnwidth}%
\begin{center}
\includegraphics[width=0.7\textwidth]{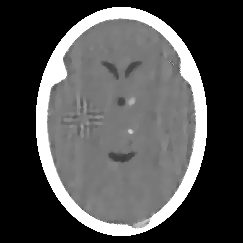}
\par\end{center}%
\end{minipage}}%
\end{minipage}
\par\end{centering}

\begin{centering}
\begin{minipage}[t]{0.5\columnwidth}%
\subfloat[BlobART, iteration 14]{\centering{}%
\begin{minipage}[t]{1\columnwidth}%
\begin{center}
\includegraphics[width=0.7\textwidth]{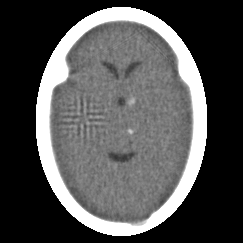}
\par\end{center}%
\end{minipage}}%
\end{minipage}%
\begin{minipage}[t]{0.5\columnwidth}%
\subfloat[Phantom]{\centering{}%
\begin{minipage}[t]{1\columnwidth}%
\begin{center}
\includegraphics[width=0.7\textwidth]{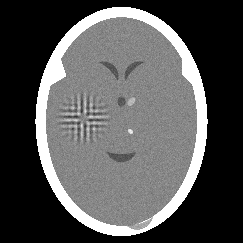}
\par\end{center}%
\end{minipage}}%
\end{minipage}
\par\end{centering}

\caption{Phantom and best images obtained by each algorithm. The best images
are those that for which (i) the parameters have been selected according
to the method described in Subsection \ref{sub:Algorithmic-parameter-selection}
and specified in Table \ref{table:optimal_pars}, and (ii) the value
of the selective error (the $SE$ of (\ref{eq:SE})) is the smallest
among all the ones plotted in Fig. \ref{fig:comparison_plot}. \label{fig:phantom_and_best_images}}
\end{figure}

\subsection{Algorithmic parameter selection\label{sub:Algorithmic-parameter-selection}}

The compared algorithms, with the exception of CG, have parameters
to be tuned. To select these parameters, we evaluated the quality
of the reconstructed image $x$ using the figure of merit called selective
error ($SE$) and then chose, for each algorithm separately, the parameters
that provided the best reconstructed image within the first $15$
iterations. $SE$ is defined by
\begin{equation}
SE(x)=C\sqrt{\sum_{(i,j)\in S}\left([x]_{i,j}-[x^{\dagger}]_{i,j}\right)^{2}}\text{,}\label{eq:SE}
\end{equation}
where $C$ is a constant dependent on the phantom $x^{\dagger}$ and
$S$ is the set of pixels contained in a centered ellipse with a $5$~cm
horizontal and a $7$~cm vertical semi-axis. This ellipse is just
inside the ``skull'' in the phantom.

All the superiorized methods share three parameters ($K$, $a$ and
$\gamma$, see Algorithm 1 and 3 ) that were drawn from the following
sets: $K\in\{10,20,40\}$, $a\in\{1-10^{-5},1-10^{-4},\dots,1-10^{-1}\}$,
$\gamma\in\{10^{-2},5\cdot10^{-2}\}$. PCG and SupPCG have two preconditioning
parameters: $\mu\in\{10^{-5},10^{-4},10^{-3}\}$ and $\rho\in\{0.4,0.6,0.8\}$;
see~(\ref{eq:preconditioning}). ART, BlobART and SupART use the
signal-to-noise ratio $r\in\{5,10\}$, which corresponds to the parameter
with the same symbol in~\cite[eq. (7)]{hhl79} (where the version
of ART we use here is described) and a relaxation parameter $\lambda\in\{10^{-2},5\cdot10^{-2},10^{-1},5\cdot10^{-1},1\}$;
see also \cite[Section 11.3]{HERMAN:2009a}. The combination of parameters,
within the sets specified above, that provided the best results are
listed in Table~\ref{table:optimal_pars}.

\begin{table}
\begin{centering}
{\footnotesize{}}%
\begin{tabular}{|c|c|}
\hline
{\footnotesize{}Method} & {\footnotesize{}Optimal Parameters}\tabularnewline
\hline
\hline
{\footnotesize{}SupCG} & {\footnotesize{}$(K,a,\gamma)=\left(40,1-10^{-5},5\cdot10^{-2}\right)$}\tabularnewline
\hline
{\footnotesize{}PCG} & {\footnotesize{}$(\mu,\rho)=\left(10^{-3},0.6\right)$}\tabularnewline
\hline
{\footnotesize{}SupPCG} & {\footnotesize{}$(K,a,\gamma,\mu,\rho)=\left(40,1-10^{-5},10^{-2},10^{-5},0.8\right)$}\tabularnewline
\hline
{\footnotesize{}ART and BlobART} & {\footnotesize{}$(r,\lambda)=\left(5,10^{-2}\right)$}\tabularnewline
\hline
{\footnotesize{}SupART} & {\footnotesize{}$(K,a,\gamma,r,\lambda)=\left(10,1-10^{-5},10^{-2},5,5\cdot10^{-2}\right)$}\tabularnewline
\hline
\end{tabular}
\par\end{centering}{\footnotesize \par}

\caption{Optimal parameters for compared algorithms.}
\label{table:optimal_pars}
\end{table}

\bigskip{}

\begin{figure}[t]
\includegraphics[width=1\columnwidth]{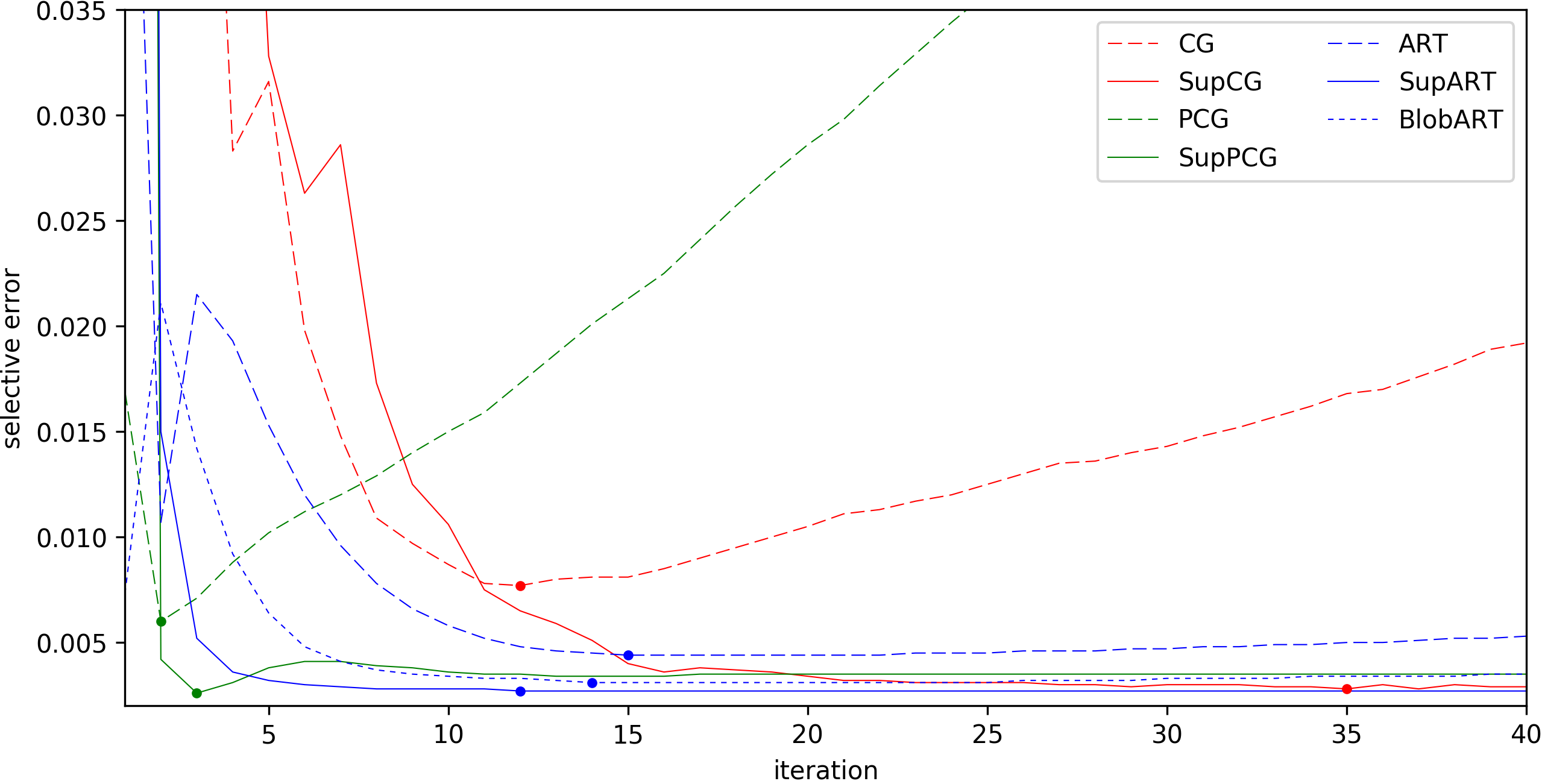}\caption{Iteration-wise comparison of $SE$ evolution among algorithms. Location
of the smallest $SE$ value on each curve is indicated.}
\label{fig:comparison_plot}
\end{figure}

\begin{figure}[t]
\includegraphics[width=1\columnwidth]{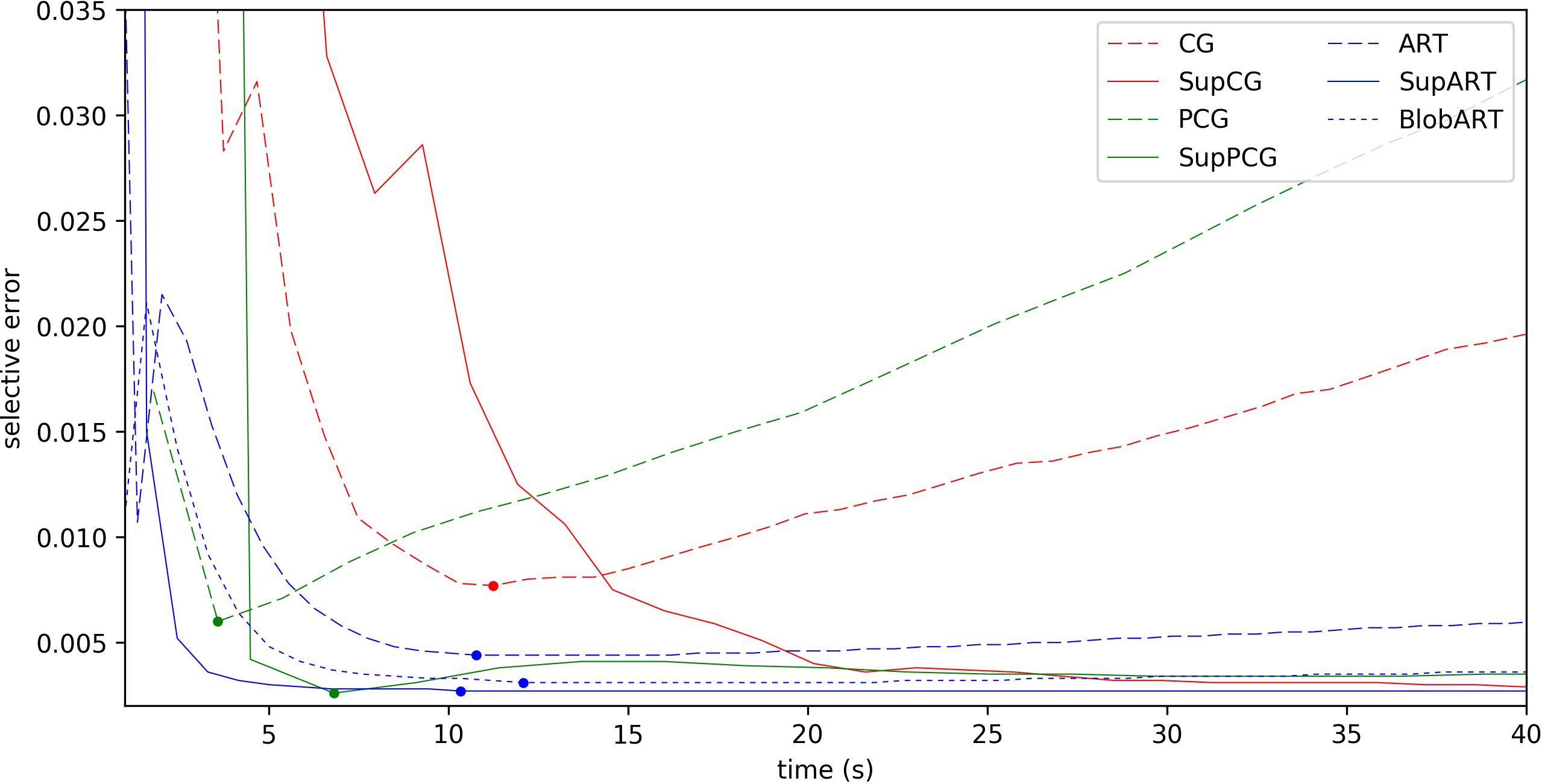}\caption{Time-wise comparison of $SE$ evolution among algorithms. Location
of the smallest $SE$ value on each curve is indicated.}
\label{fig:comparison_plot_time}
\end{figure}

\subsection{Numerical results}

In Fig. \ref{fig:comparison_plot} we present the evolution of the
$SE$ of (\ref{eq:SE}) as iterations proceed, while Figure~\ref{fig:comparison_plot_time}
exhibits a time-wise comparison among the methods. The computer used
to run the methods had an Intel i7 7700HQ processor running at up
to 3.4GHZ with 32GB of RAM available for the computations. We notice
that a superiorized method obtains better values for the $SE$ figure
of merit than the corresponding unsuperiorized version for almost
every iteration. This does not come as a surprise, for two reasons.
First, the superiorized version is supposed to incorporate better
prior knowledge about the sought-after image to be reconstructed through
the secondary criterion. Second, in the experiments we present, the
algorithmic parameter settings were tailored to obtain an improved
value of the figure of merit $SE$.

\section{Theoretical Analysis\label{sec:Theorethical-Analysis}}

In the present section we prove that the SupPCG method (Algorithm~\ref{algo:SPCG})
terminates under a suitable hypothesis on the superiorization sequence
$\left\{ s_{k}\right\} $ that is defined by calling $\mathrm{S}_{TV}(x_{k},a,\gamma,K)$.
Our proof is based on facts established in~\cite{zlh17}, which contains
a proof that SupCG terminates. The needed relation between the SupPCG
and the SupCG is established via the so-called Transformed PCG algorithm
(TPCG) and its superiorized version (SupTPCG). Next we discuss these
algorithms.

Since the preconditioning matrix $M$ is symmetric positive-definite,
there exists an invertible matrix $N$ such that $M=N^{T}N$ and,
for such a matrix, the eigenvalues of $MA$ are the same as the eigenvalues
of $NAN^{T}$ (if $v$ is an eigenvector of $MA$, then $N^{-T}v$
is an eigenvector of $NAN^{T}$ with the same associated eigenvalue).
Applying one step of the CG algorithm (which is exactly Algorithm~\ref{algo:PCG_step}
with $M$ the identity matrix) to the linear system $NAN^{T}\hat{x}=Ny$
provides one step of TPCG, as given in Algorithm~\ref{algo:TPCG_step}.
This step is also called from SupTPCG, see Step 13 of Algorithm~\ref{algo:STPCG}.

There is a strict equivalence between Algorithm~\ref{algo:SPCG}
and Algorithm~\ref{algo:STPCG}. To see this consider Table \ref{3-5_eqivalence}.
We claim, and this can be verified by following step-by-step SupPCG$\left(x_{0},a,\gamma,K,\varepsilon\right)$
and SupTPCG$\left(N^{-T}x_{0},a,\gamma,K,\varepsilon\right)$, that
the behavior of the two algorithms are identical via the stated equivalences.
In particular, since $f\left(x_{k-\nicefrac{1}{2}}\right)=f\left(N^{T}\hat{x}_{k-\nicefrac{1}{2}}\right)$,
the two algorithms terminate for the same value of $k$ and the two
images returned by the respective algorithms can be obtained from
each other using $x_{k+\nicefrac{1}{2}}=N^{T}\hat{x}_{k+\nicefrac{1}{2}}$.

\begin{algorithm}[t]
\caption{$\mathrm{U}_{\text{TPCG}}$$\left(\hat{x},\hat{p},\hat{h}\right)$}

\label{algo:TPCG_step}

\begin{algorithmic}[1]

\STATE{\textbf{$\hat{g}\leftarrow NAN^{T}\hat{x}-Ny$}}

\STATE{\textbf{$\hat{z}\leftarrow\hat{g}$}}

\STATE{\textbf{$\beta\leftarrow\hat{z}{}^{T}\hat{h}/\hat{p}^{T}\hat{h}$}}

\STATE{\textbf{$\hat{p}\leftarrow-\hat{z}+\beta\hat{p}$}}

\STATE{$\hat{h}\leftarrow NAN^{T}\hat{p}$}

\STATE{\textbf{$\alpha\leftarrow-\hat{g}^{T}\hat{p}/\hat{p}^{T}\hat{h}$}}

\STATE{\textbf{$\hat{x}\leftarrow\hat{x}+\alpha\hat{p}$}}

\STATE{\textbf{return }$\left(\hat{x},\hat{p},\hat{h}\right)$}

\end{algorithmic}
\end{algorithm}

\begin{algorithm}[t]
\caption{SupTPCG$\left(\hat{x}_{0},a,\gamma,K,\varepsilon\right)$}

\label{algo:STPCG}

\normalsize{

\begin{algorithmic}[1]

\STATE{\textbf{$\hat{x}_{\nicefrac{1}{2}}\leftarrow\hat{x}_{0}$}}

\STATE{\textbf{$\hat{g}_{0}\leftarrow NAN^{T}\hat{x}_{0}-Ny$}}

\STATE{\textbf{$\hat{z}_{0}\leftarrow\hat{g}_{0}$}}

\STATE{\textbf{$\hat{p}_{0}\leftarrow-\hat{z}_{0}$}}

\STATE{\textbf{$\hat{h}_{0}\leftarrow NAN^{T}\hat{p}_{0}$}}

\STATE{\textbf{$\alpha\leftarrow-\hat{g}_{0}^{T}\hat{p}_{0}/\hat{p}_{0}^{T}\hat{h}_{0}$}}

\STATE{\textbf{$\hat{x}_{1}\leftarrow\hat{x}_{0}+\alpha\hat{p}_{0}$}}

\STATE{\textbf{$k\leftarrow1$}}

\STATE{\textbf{$\ell_{1}\leftarrow0$}}

\STATE{\textbf{$\textbf{while}$ $f\left(N^{T}\hat{x}_{k-\nicefrac{1}{2}}\right)>\varepsilon$}}

\STATE{$\quad$\textbf{$\left(s_{k},\ell_{k+1}\right)\leftarrow\mathrm{S}_{TV}\left(N^{T}\hat{x}_{k},\ell_{k},a,\gamma,K\right)$}}

\STATE{\textbf{$\mbox{\quad}\hat{x}_{k+\nicefrac{1}{2}}\leftarrow\hat{x}_{k}+N^{-T}s_{k}$}}

\STATE{$\quad$$\left(\hat{x}_{k+1},\hat{p}_{k},\hat{h}_{k}\right)\leftarrow\mathrm{U}_{\text{TPCG}}\left(\hat{x}_{k+\nicefrac{1}{2}},\hat{p}_{k-1},\hat{h}_{k-1}\right)$}

\STATE{$\quad$$k\leftarrow k+1$}

\STATE{$\textbf{return }\left(k,\hat{x}_{k-\nicefrac{1}{2}}\right)$}

\end{algorithmic}}
\end{algorithm}

\begin{table}
\begin{tabular}{|c|c|}
\hline
Algorithm 3 & Algorithm 5\tabularnewline
\hline
\hline
$x_{i}=N^{T}\hat{x}_{i}$ & $\hat{x}_{i}$\tabularnewline
\hline
$g_{0}$ & $\hat{g}_{0}=Ng_{0}$\tabularnewline
\hline
$z_{0}=N^{T}\hat{z}_{0}$ & $\hat{z}_{0}$\tabularnewline
\hline
$p_{k}=N^{T}\hat{p}_{k}$ & $\hat{p}_{k}$\tabularnewline
\hline
$h_{k}$ & $\hat{h}_{k}=Nh_{k}$\tabularnewline
\hline
\end{tabular}\caption{Equivalences of the variables in Algorithm~\ref{algo:SPCG} and Algorithm~\ref{algo:STPCG}.\textcolor{black}{{}
The $i$ in the table can be $k$, $k-\nicefrac{1}{2}$, or $k+\nicefrac{1}{2}$
for an integer $k$.}}

\label{3-5_eqivalence}
\end{table}

Next we show the equivalence of Algorithm \ref{algo:STPCG} and Algorithm
7 of \cite{zlh17} under some assumptions. To make the current paper
self-contained, we reproduce here (as Algorithm \ref{algo:SupCG})
Algorithm 7 of \cite{zlh17}. We keep the boldface notation of~\cite{zlh17}
in order to avoid confusion with the symbols used in the algorithms
introduced in the present paper. Consistently with the notation of
that previous paper, the function $f'$ in Algorithm 7 is defined
by
\begin{align}
f'(\mathbf{x}) & =\frac{1}{2}\|\mathbf{A}\mathbf{x}-\mathbf{y}\|^{2},\label{eq: f'(x)}
\end{align}
where $\mathbf{\mathit{\mathbf{A}}}=RN^{T}$ and $\mathbf{y}=b$.

As already defined in Subsection \ref{sub:Superiorized-Preconditioned-Conj},
$A=R^{T}R$ and $y=R^{T}b$. By setting $\mathbf{\mathit{\mathbf{A}}}$
and $\mathbf{\mathbf{y}}$ in Algorithm 8 of \cite{zlh17} (restated
here as Algorithm~\ref{algo:UCG_step}) to be as we have just defined
them, we see that Algorithm~\ref{algo:TPCG_step} is equivalent to
Algorithm 8 of \cite{zlh17} in the sense that if we have in the input
of the two algorithms the equivalences $\mathbf{x}=\hat{x}$, $\mathbf{p}=\hat{p}$,
and $\mathbf{h}=\hat{h}$, then $\mathbf{g}=\hat{g}$ (to see this,
notice that in Algorithm~\ref{algo:TPCG_step} we have $Ny=NR^{T}b=\mathbf{A}^{T}\mathbf{y}$,
and $NAN^{T}=NR^{T}RN^{T}=\mathbf{A}^{T}\mathbf{A}$), the $\beta$
in the two algorithms are equal (since $\hat{z}=\hat{g}$) and, therefore,
we also have that the output values satisfy $\mathbf{p}=\hat{p}$,
$\mathbf{h}=\hat{h}$ (because $NAN^{T}=\mathbf{A}^{T}\mathbf{A}$)
and $\mathbf{x}=\hat{x}$ (since, as is easily verified, the value
of $\alpha$ is the same for the two algorithms). Thus, for identical
inputs, Algorithms 4 and 6 will return identical outputs.

\begin{algorithm}[t]
\caption{$\mathrm{U}_{\text{CG}}$$(\mathbf{x},\mathbf{p},\mathbf{h})$}

\label{algo:UCG_step}

\begin{algorithmic}[1]

\STATE{\textbf{$\mathbf{g}\leftarrow\mathbf{A}^{T}(\mathbf{A\mathbf{x}}-\mathbf{y})$}}

\STATE{\textbf{$\beta\leftarrow\mathbf{g}{}^{T}\mathbf{h}/\mathbf{p}^{T}\mathbf{h}$}}

\STATE{\textbf{$\mathbf{p}\leftarrow-\mathbf{g}+\beta\mathbf{p}$}}

\STATE{$\mathbf{h}\leftarrow\mathbf{A}^{T}\mathbf{A}\mathbf{p}$}

\STATE{\textbf{$\alpha\leftarrow-\mathbf{g}{}^{T}\mathbf{p}/\mathbf{p}^{T}\mathbf{h}$}}

\STATE{\textbf{$\mathbf{x}\leftarrow\mathbf{x}+\alpha\mathbf{p}$}}

\STATE{\textbf{return }$(\mathbf{x},\mathbf{p},\mathbf{h})$}

\end{algorithmic}
\end{algorithm}

\begin{algorithm}[t]
\caption{SupCG$\left(\mathbf{x}_{0},a,\gamma,K,\varepsilon'\right)$}

\label{algo:SupCG}

\begin{algorithmic}[1]

\STATE{\textbf{$\mathbf{x}_{\nicefrac{1}{2}}\leftarrow\mathbf{x}_{0}$}}

\STATE{\textbf{$\mathbf{g}_{0}\leftarrow\mathbf{A}^{T}(\mathbf{A\mathbf{x}}_{0}-\mathbf{y})$}}

\STATE{\textbf{$\mathbf{p}_{0}\leftarrow-\mathbf{g}_{0}$}}

\STATE{$\mathbf{h}_{0}\leftarrow\mathbf{A}^{T}\mathbf{A}\mathbf{p}_{0}$}

\STATE{\textbf{$\alpha\leftarrow-\mathbf{g}{}_{0}^{T}\mathbf{p}_{0}/\mathbf{p}_{0}^{T}\mathbf{h}_{0}$}}

\STATE{\textbf{$\mathbf{x}_{1}\leftarrow\mathbf{x}_{0}+\alpha\mathbf{p}_{0}$}}

\STATE{\textbf{$k\leftarrow1$}}

\STATE{\textbf{$\textbf{while}$ $f'(\mathbf{x}_{k-\nicefrac{1}{2}})>\varepsilon'$}}

\STATE{\textbf{$\mbox{\quad}\mathbf{x}_{k+\nicefrac{1}{2}}\leftarrow\text{\textnormal{perturbed}}\left(\mathbf{x}_{k}\right)$}}

\STATE{$\quad$$\left(\mathbf{x}_{k+1},\mathbf{p}_{k},\mathbf{h}_{k}\right)\leftarrow\mathrm{U}_{\text{PCG}}\left(\mathbf{x}_{k+\nicefrac{1}{2}},\mathbf{p}_{k-1},\mathbf{h}_{k-1}\right)$}

\STATE{$\quad$$k\leftarrow k+1$}

\STATE{$\textbf{return }\left(k,\mathbf{x}_{k-\nicefrac{1}{2}}\right)$}

\end{algorithmic}
\end{algorithm}

To show the equivalence of Algorithm \ref{algo:STPCG} and Algorithm
\ref{algo:SupCG}, we associate $\mathbf{\mathit{\mathbf{x}}}_{k}$,
$\mathbf{\mathit{\mathbf{x}}}_{k+\nicefrac{1}{2}}$, $\mathbf{\mathbf{g}}_{k}$,
$\mathbf{\mathbf{h}}_{k}$, $\mathbf{p}_{k}$ in Algorithm~\ref{algo:SupCG}
with $\hat{x}_{k}$, $\hat{x}_{k+\nicefrac{1}{2}}$, $\hat{g}_{k}$,
$\hat{h}_{k}$, $\hat{p}_{k}$ in Algorithm~\ref{algo:STPCG} by
$\mathbf{\mathit{\mathbf{x}}}_{k}=\hat{x}_{k}$, $\mathbf{x}_{k+\nicefrac{1}{2}}=\hat{x}_{k+\nicefrac{1}{2}}$,
$\mathbf{g}_{k}=\hat{g}_{k}$, $\mathbf{\mathbf{h}}_{k}=\hat{h}_{k}$,
$\mathbf{\mathbf{p}}_{k}=\hat{p}_{k}$. We now compare the step-by-step
executions of SupTPCG$\left(\hat{x}_{0},a,\gamma,K,\varepsilon\right)$
and SupCG$\left(\mathbf{x}_{0},a,\gamma,K,\varepsilon'\right)$, with
$\mathbf{\mathit{\mathbf{x}}}_{0}=\hat{x}_{0}$ and $\varepsilon'=\frac{\varepsilon}{2}$.
Clearly, after Step 1 in the two algorithms $\mathbf{\mathit{\mathbf{x}}}_{\nicefrac{1}{2}}=\hat{x}_{\nicefrac{1}{2}}$.
That $\mathbf{g}_{0}=\hat{g}_{0}$ follows from the facts that $Ny=\mathbf{A}^{T}\mathbf{y}$,
and $NAN^{T}=\mathbf{A}^{T}\mathbf{A}$. Continuing in this fashion
it is trivial to check that just before entering the $\mathbf{while}$
statement for the first time in the two algorithms the values of the
associated vectors match as stated above.

We now consider the condition that appears in the $\mathbf{while}$
statement in the algorithms. We claim that if just before entering
the $\mathbf{while}$ in the two algorithms the values of the associated
vectors match as stated above, then the condition is either satisfied
in both algorithms or is not satisfied in both algorithm. Indeed,
due to the definition of $f$ as the squared error and (\ref{eq: f'(x)}),
\[
f\left(N^{T}\hat{x}_{k-\nicefrac{1}{2}}\right)=\left\Vert RN^{T}\hat{x}_{k-\nicefrac{1}{2}}-b\right\Vert ^{2}=\left\Vert \mathbf{A}\mathbf{x}_{k-\nicefrac{1}{2}}-\mathbf{y}\right\Vert ^{2}=2f'\left(\mathbf{x}_{k-\nicefrac{1}{2}}\right).
\]
Since $\varepsilon'=\frac{\varepsilon}{2}$, our claim on the satisfaction
of the conditions is valid.

We specify the perturbation operator in Step 9 of Algorithm~\ref{algo:SupCG}
by
\[
\mathrm{perturbed}\left(\mathbf{x}_{k}\right)=\mathbf{x}_{k}+N^{-T}s_{k},
\]
where $s_{k}$ is the same as given by Algorithm~\ref{algo:STPCG}.
Using the above associations and facts, we see that the behavior of
Algorithm \ref{algo:STPCG} is equivalent to that of Algorithm~\ref{algo:SupCG},
and, therefore, to that of Algorithm~7 of \cite{zlh17}.

Theorem A.1 of \cite{zlh17} says that SupCG$\left(\mathbf{x}_{0},a,\gamma,K,\varepsilon'\right)$
terminates provided that $\varepsilon'>\varepsilon_{0}'$, where
\[
\varepsilon_{0}'=\min_{\mathbf{x}}\frac{1}{2}\|\mathbf{A}\mathbf{x}-\mathbf{y}\|^{2}.
\]
Recalling the definitions of $\mathbf{A}$ and $\mathbf{y}$ after
(\ref{eq: f'(x)}) and the fact that $N$ is invertible, we get that
$\varepsilon'_{0}=\frac{1}{2}\varepsilon_{0,}$ for
\[
\varepsilon_{0}=\min_{\hat{x}}\left\Vert RN^{T}\hat{x}-b\right\Vert ^{2}=\min_{x}\left\Vert Rx-b\right\Vert ^{2}.
\]
This together with the comparison of the the step-by-step executions
of SupTPCG$\left(\hat{x}_{0},a,\gamma,K,\varepsilon\right)$ and SupCG$\left(\mathbf{x}_{0},a,\gamma,K,\varepsilon'\right)$,
with $\mathbf{\mathit{\mathbf{x}}}_{0}=\hat{x}_{0}$ and $\varepsilon'=\frac{\varepsilon}{2}$,
gives us our promised termination result for SupPCG:
\begin{thm}
Given any positive number $\varepsilon$ such that $\varepsilon>\varepsilon_{0}$,
with
\begin{equation}
\varepsilon_{0}=\min_{x}\left\Vert Rx-b\right\Vert ^{2},\label{eq: epsilon_0}
\end{equation}
Algorithm~\ref{algo:SPCG} terminates within a finite number of iterations.
\end{thm}

\section{Concluding remarks}

\label{sec:conclusions-1}

We have discussed the superiorized version of the Preconditioned Conjugate
Gradient method (SupPCG) for tomographic image reconstruction. Experimental
work has been presented indicating that, when compared to several
other methods, SupPCG produces images of good quality within a short
time. Furthermore, we have proved that the algorithm produces, in
a finite number of steps, an image whose data-inconsistency is no
worse than what is specified by the user, provided that such an image
exists.

\section*{Acknowledgments}

The research of E. S. Helou was funded by FAPESP grants 2013/07375-0
and 2016/24286-9. The research of M. V. W. Zibetti was partially supported
by CNPq grant 475553/2013-6. The research of C. Lin was partially
supported by China Scholarship Counci\textcolor{black}{l. We are grateful
to Christoph Scharf for essential help with programming the PCG algorithm
in the SNARK14 environment.}


\begin{thebibliography}{10}
\bibitem{chj17} Y.~Censor, G.~T. Herman, and M.~Jiang. \newblock Superiorization: theory and applications. \newblock {\em Inverse Problems}, 33, 040301, 2017.
\bibitem{com96} P.~L. Combettes. \newblock The convex feasibility problem in image recovery. \newblock {\em Advances in Imaging and Electron Physics}, 95:155--270, 1996.
\bibitem{gah14} E.~Gardu{\~{n}}o and G.~T. Herman. \newblock Superiorization of the {ML-EM} algorithm. \newblock {\em IEEE Transactions on Nuclear Science}, 61:162--172, 2014.
\bibitem{gh17} E.~Gardu{\~{n}}o and G.~T. Herman. \newblock Computerized tomography with total variation and with shearlets. \newblock {\em Inverse Problems}, 33, 044011, 2017.
\bibitem{hzm16} E.~S. Helou, M.~V.~W. Zibetti, and E.~X. Miqueles. \newblock Superiorization of incremental optimization algorithms for   statistical tomographic image reconstruction. \newblock {\em Inverse Problems}, 33, 044010, 2017.
\bibitem{HERMAN:2009a} G.~T. Herman. \newblock {\em Fundamentals of Computerized Tomography: {I}mage Reconstruction   from Projections}. \newblock Springer, London, UK, 2nd. edition, 2009.
\bibitem{Herman2012} G.~T. Herman, E.~Gardu{\~{n}}o, R.~Davidi, and Y.~Censor. \newblock Superiorization: An optimization heuristic for medical physics. \newblock {\em Medical Physics}, 39:5532--5546, 2012.
\bibitem{hhl79} G.~T. Herman, H.~Hurwitz, A.~Lent, and H.-P. Lung. \newblock On the {B}ayesian approach to image reconstruction. \newblock {\em Information and Control}, 42:60--71, 1979.
\bibitem{KLUK13} J.~Klukowska, R.~Davidi, and G.~T. Herman. \newblock {SNARK09} - {A} software package for the reconstruction of {2D}   images from {1D} projections. \newblock {\em Computer Methods and Programs in Biomedicine}, 110:424--440,   2013.
\bibitem{mhc98} R.~Marabini, G.~T. Herman, and J.~M. Carazo. \newblock {3D} reconstruction in electron microscopy using {ART} with smooth   spherically symmetric volume elements (blobs). \newblock {\em Ultramicroscopy}, 72:53--65, 1998.
\bibitem{Vogel-2002} C.~R. Vogel. \newblock {\em {Computational Methods for Inverse Problems}}. \newblock SIAM, Philadelphia, PA, 2002.
\bibitem{zlh17} M.~V.~W. Zibetti, C.~Lin, and G.~T. Herman. \newblock Total variation superiorized conjugate gradient method for image   reconstruction. \newblock {\em Inverse Problems}, 34, 034001, 2018.
\end{thebibliography}
\end{document}